\documentclass[12pt, reqno]{amsart}
\usepackage{amsmath, amsthm, amscd, amsfonts, amssymb, graphicx, color}
\usepackage[bookmarksnumbered, colorlinks, plainpages]{hyperref}
\hypersetup{colorlinks=true,linkcolor=red, anchorcolor=green, citecolor=cyan, urlcolor=red, filecolor=magenta, pdftoolbar=true}

\newtheorem{theorem}{Theorem}[section]
\newtheorem{lemma}[theorem]{Lemma}
\newtheorem{proposition}[theorem]{Proposition}
\newtheorem{corollary}[theorem]{Corollary}
\theoremstyle{definition}
\newtheorem{definition}[theorem]{Definition}
\newtheorem{example}[theorem]{Example}

\theoremstyle{remark}
\newtheorem{remark}[theorem]{Remark}
\numberwithin{equation}{section}

\begin{document}
\setcounter{page}{1}

\title[Harnack type inequality]
{Harnack type inequalities for operators in logarithmic submajorisation}

\author[Y. Han]{Yazhou Han}

\address{College of Mathematics and Systems Science, Xinjiang
University, Urumqi 830046, China}
\email{\textcolor[rgb]{0.00,0.00,0.84}{hanyazhou@foxmail.com}}

\author[C. Yan]{Cheng Yan}
\address{College of Mathematics and Systems Science, Xinjiang
University, Urumqi 830046, China}
\email{\textcolor[rgb]{0.00,0.00,0.84}{yanchengggg@163.com}}



\begin{abstract} The aim of this paper is to study the Harnack type
logarithmic submajorisation and Fuglede-\,-Kadison determinant
inequalities for operators in a  finite von Neumann algebra.
In particular, the Harnack type determinant inequalities due to
Lin-\,-Zhang\cite{LZ2017} and Yang-\,-Zhang\cite{YZ2020} are
extended to the case of operators in a finite von Neumann algebra.
\end{abstract}
\maketitle

\section{Introduction}
The classical Harnack inequality, named after
Carl Gustav Axel von Harnack,
gives an estimate from above and an estimate from
below for a positive harmonic function in a domain.
Even though the classical Harnack inequality is almost trivially derived from the Poisson
formula, the consequences that may be deduced from Harnack inequality are particularly of great importance.
Later, these inequalities became an important tool in the general theory of
harmonic functions and partial differential equations.
There exist as yet extensive works on
generalized Harnack inequalities in various forms,
see \cite{M2007, W2013, YZ2020} for a nice introduction about the inequality.
The purpose of this paper is to investigate the Harnack type determinant inequality for operators and matrices.

With the help of Lagrange multiplier method, the following Harnack type
determinant inequality was established by Tung\cite{T1964}, as a tool to study Harnack inequality:
If $Z\in \mathbb{M}_n$ is a complex matrix with singular values
$r_k$ with $0\leq r_k<1, k=1, 2, ..., n$, then
\begin{equation}\label{inequ. T1}
\prod_{k=1}^n\frac{1-r_k}{1+r_k}\leq\frac{\det(\mathbb{I}-Z^*Z)}{|\det(\mathbb{I}-UZ)|^2}\leq
\prod_{k=1}^n\frac{1+r_k}{1-r_k}, U\in \mathbb{U}_n,
\end{equation} where $\mathbb{U}_n$ denotes the set of all $n\times n$ unitary
matrices $U$.
From these bounds Tung obtained upper and lower
bounds of a Poisson kernel on $\mathbb{U}_n$(see \cite{T1964}), hence that the so-called Harnack's
first and second theorems are established. Tung's work drew immediate attention
of Hua and Marcus.
Using majorisation theory and singular value (eigenvalue)
inequalities of Weyl, Marcus \cite{M1965} gave another proof of (\ref{inequ. T1})
and gave an equivalent form of (\ref{inequ. T1}).
 Almost at the same time,
a proof of (\ref{inequ. T1}) was also given by Hua\cite{H1965} based on the
determinantal inequality he had previously obtained in\cite{H1955}.
In the past decades,
Tung's work has attracted attentions of mathematicians and been extended to various setting
(see
 \cite{JL2020, LZ2017, M2007, W2013, YZ2020} and the references therein
for more details).
Among these outstanding works we will be interested in
Lin-\,-Zhang's and Yang-\,-Zhang's work.
Specifically, with $A=UZ$, (\ref{inequ. T1}) is equivalently
rewritten in terms of eigenvalues (\cite{JL2020, YZ2020}) as
\begin{equation}\label{inequ. YZ1}
\prod_{k=1}^n\frac{1-r_k}{1+r_k}\leq\prod_{k=1}^n
\lambda_k((\mathbb{I}-A^*)^{-1}(\mathbb{I}-A^*A)(\mathbb{I}-A)^{-1})
\leq\prod_{k=1}^n\frac{1+r_k}{1-r_k},
\end{equation}
where $Z\in \mathbb{M}_n$ is a complex matrix with singular values
$r_k$ with $0\leq r_k<1, k=1, 2, ..., n$ and $U\in \mathbb{U}_n$.
(\ref{inequ. YZ1}) leads to the study of inequalities of
logarithmic submajorisation  of
eigenvalues and singular values. Following this line,
an interesting generalization of (\ref{inequ. YZ1}) is
presented by Yang-\,-Zhang\cite{YZ2020} and
Jiang-\,-Lin\cite{JL2020} as follows:
\begin{equation}\label{FK det 1}
\prod_{k\in K}
\lambda_k((\mathbb{I}-A^*)^{-1}(\mathbb{I}-A^*A)(\mathbb{I}-A)^{-1})
 \leq\prod_{k\in K}\frac{1+r_k}{1-r_k},
\end{equation}
\begin{equation}\label{FK det 2}
\prod_{i\in K}
\lambda_{n-k+1}((\mathbb{I}-A^*)^{-1}(\mathbb{I}-A^*A)(\mathbb{I}-A)^{-1})
\geq\prod_{k\in K}(1-r_k^2)\prod_{i=1}^{|K|}\frac{1}{(1+r_i)^2},
\end{equation}
where $K$ is a subset of $\{r_1, r_2, \cdots, r_n\}$
and  $|K|$ denotes the number of terms in $K$.
The main theme of the paper is to continue with Jiang-\,-Lin and Yang-\,-Zhang's work
and to show their results hold in the case of operators in finite von Neumann algebras.

We are concerned with the Harnack type logarithmic submajorisation
inequality and
Fuglede-Kadison determinant
inequality for operators in a finite von Neumann algebra.
The properties of the logarithmic submajorisation and
Fuglede-\,-Kadison determinant for operators in a finite
von Neumann algebra was investigated by many authors, see
for example \cite{B1983, BL2008, HSZ2020}.
Those properties are important, for example,
in investigation of noncommutative Hardy spaces
and invariant subspaces for operators in von Neumann algebras.
By adapting the techniques in \cite{YZ2020, FK1986, N1987}, we obtain some
 inequalities which is related to the
Harnack type logarithmic submajorisation inequality
and Fuglede-\,-Kadison determinant inequality. In particular,
we show that the inequalities (\ref{FK det 1}) and (\ref{FK det 2}) hold for
 operators in a finite von Neumann algebra.
We will conclude this paper with a series of logarithmic submajorisation
 inequalities which is related to  Cayley transform.

\section{Preliminaries}
\subsection{von Neumann algebras}
Suppose that $\mathcal{H}$ is a separable Hilbert space over the field $\mathbb{C}$ and
$\mathbb{I}$ is the identity operator in $\mathcal{H}$. We will denote by
  $\mathcal{B}(\mathcal{H})$ the $*$-algebra of all linear bounded operators in $\mathcal{H}$.
Let $\mathcal{M}$ be a $*$-subalgebra of
$\mathcal{B}(\mathcal{H})$ containing the identity operator $\mathbb{I}$.
Then $\mathcal{M}$ is called a von Neumann algebra
if $\mathcal{M}$ is weak operator closed.
Let $\mathcal{M}^+$ denote the positive part of $\mathcal{M}$.
We recall that a weight on $\mathcal{M}$
is a map $\tau: \mathcal{M}^+\rightarrow [0, \infty]$ satisfying
 \begin{enumerate}
\item $\tau(x+y)=\tau(x)+\tau(y),$ for all $x, y\in \mathcal{M}^+$;
\item $\tau(\alpha x)=\alpha\tau(x)$ for all $x\in \mathcal{M}^+$ and $\alpha\in [0, \infty)$,
with the convention $0\cdot\infty=0.$
\end{enumerate}

The weight $\tau$ is called faithful if $\tau(x^*x)=0$ implies $x=0$, normal if
$x_i\uparrow_i x$ in $\mathcal{M}^+$ implies that $0\leq\tau(x_i)\uparrow_i\tau(x)$,
tracial if $\tau(x^*x)=\tau(xx^*)$ for all $x\in\mathcal{M}$.
Note that since $(x_i)$ is bounded there is $x$ in $\mathcal{M}^+$ such
that, for any $h$ in $\mathcal{H}$, $\langle x_ih, h\rangle\uparrow\langle xh, h\rangle$,
which implies that $x_i$ tends to $x$ weak*
and hence $x\in \mathcal{M}^+$.
The operator $x$ is obviously the least upper bound of $(x_i)$,
it is natural to denote it by $\sup_i x_i$.
The self-adjoint part of $\mathcal{M}$, $\mathcal{M}^{sa}$, is a partially ordered vector space under the
ordering $x\geq0$ defined by $\langle x\xi, \xi\rangle\geq0, \xi\in \mathcal{H}$.
Recall that $x\in\mathcal{M}$
is contractive if $\|x\|\leq1$ and strictly contractive if $\|x\|<1$.
Moreover, if $x$ is strict contractive, then $\mathbb{I}-x^*x$ is invertible and $\mathbb{I}-x^*x\geq0$.

It is also customary to say trace instead of tracial weight.
A trace $\tau$ is called finite if $\tau(\mathbb{I})<\infty.$
A finite trace $\tau$ is extended uniquely to a positive
linear functional on $\mathcal{M}$ which will also
be denoted by $\tau.$ A positive linear functional
$\tau$ on a von Neumann algebra is said to be a state if $\tau(\mathbb{I})=1$.

A von Neumann algebra $\mathcal{M}$ is called finite if the family formed of the finite normal traces
separates the points of $\mathcal{M}$. Clearly this happens if $\mathcal{M}$ admits a
single faithful normal finite trace.
But a finite $\mathcal{M}$ may fail to have any faithful finite trace,
for instance $\mathcal{M} = \ell^\infty(\mathbb{R})$ where $\mathbb{R}$ is equipped with counting measure.
However, on a separable Hilbert space (i.e. if $\mathcal{M}$ is weak*-separable) the converse is also true i.e.,
$\mathcal{M}$ is finite if and only if it admits a faithful normal finite trace.

In what follows, we will keep all previous notations
throughout the paper, and $\mathcal{M}$  will always denote a finite von
Neumann algebra acting on a separable Hilbert space $\mathcal{H}$,
with a normal faithful finite tracial state $\tau$, i.e.,
a  normal faithful finite trace $\tau$ satisfies that $\tau(\mathbb{I})=1$.
We refer to \cite{T1979} for von Neumann algebras.

\subsection{The eigenvalue function and generalized singular value function}
\begin{definition}
Let $x\in \mathcal{M}$ and $t>0.$
The ``$t$-th singular number(or generalized singular number)
 of $x$" $\mu_t(x)$ is defined by
$$\mu_t(x)=\inf\{\|xe\|: e~\mbox{is a projection in}~
\mathcal{M}~\mbox{with}~\tau(e^\bot)\leq t\}.$$
\end{definition}
We denote simply
by $\mu(x)$ the function $t\rightarrow\mu_t(x)$.
The generalized singular number function $t\rightarrow\mu_t(x)$ is decreasing right-continuous.
For convenience to discuss the properties of $\mu_t(x)$ we define $\mu_t^\ell(x)$ by
\[\mu_t^\ell(x)=\inf\{\|xe\|: e~\mbox{is a projection in}~
\mathcal{M}~\mbox{with}~\tau(e^\bot)< t\}.\]
Then  $t\rightarrow\mu_t^\ell(x)$ is non-increasing and right-continuous
and $\mu_t^\ell(x)=\mu_t(x)$ holds for almost every $t\in [0, 1]$.
See \cite{DPS2019, FK1986, O19701, O1970} for basic properties and detailed information
on $\mu_t(x)$ and $\mu_t^\ell(x)$.

If $x$ is self-adjoint and
$x=\int_{-\infty}^\infty tde_t(x)\in \mathcal{M}$ is the
spectral resolution of $x$ then
for any Borel subset $B\subseteq (-\infty, \infty)$ we denote by $e_B(x)$ the
corresponding spectral projection. However, we write $e_s(x)=e_{(-\infty, s]}(x).$
Given $x\in \mathcal{M}^{sa}$,
the spectral scale $\lambda_t(x)$
on $(0, \tau(\mathbb{\mathbb{I}}))$ is defined by
\[\lambda_t(x)=\inf\{s\in \mathbb{R}: \tau(\mathbb{I}-e_s(x))\leq t\}.\]
Obviously, if $0\leq x\in \mathcal{M}$
then $\lambda_t(x) = \mu_t(x)$ for $0<t<1.$
The spectral scale $\lambda_t(x)$
is non-increasing and right-continuous.
For the properties of $\lambda_t(\cdot)$,
it is important to note that $\lambda_t(x+a \mathbb{I})=\lambda_t(x)+a$
for every $x\in  \mathcal{M}^{sa}$ and $a\in \mathbb{R}$.
This property enables us to deduce estimations for $\lambda_t(x)$ from formulas on $\mu_t(x)$.

 To achieve our main results, we state some properties of $\lambda(\cdot)$ and $\mu(\cdot)$ without proof(see \cite{H1987, FK1986}).
\begin{proposition}\label{proposition 2.2}(see \cite{H1987, FK1986})
Let $x, y\in  \mathcal{M}$ and $v\in\mathcal{M}$. Then
\begin{enumerate}
\item $\mu(|x|)=\mu(x)=\mu(x^*)$ and $\mu(\alpha x)=|\alpha|\mu_t(x),$ for $t>0$ and $\alpha\in \mathbb{C}$.
\item Let $f$ be a bounded continuous
 increasing function on $[0, \infty)$ with $f(0)=0$.
Then $\mu(f(x))=f(\mu(x))$ and $\tau(f(x))=\int_0^{\tau(1)} f(\mu_t(x))dt.$
\item $\mu_{s+t}(x+y)\leq\mu_t(x)+\mu_s(y), s, t>0.$
\item If $0\leq x\leq y$, then $\mu_t(x)\leq\mu_t(y)$.
\item $\mu_{t+s}(xy)\leq\mu_t(x)\mu_s(y), s, t>0.$
\item If $x, y$ are self-adjoint, then
$\lambda_{t+s}(x+y)\leq \lambda_t(x)+\lambda_s(y), t, s\geq0, t+s\leq1.$
\item If $0\leq t\leq1$ and $x, y$ are self-adjoint, then $\lambda_t(x)\geq0$ implies that
$\lambda_t(v^*av)\leq\|v\|^2\lambda_t(x)$.
\item If $x, y$ are self-adjoint and $x\leq y$, then $\lambda_t(x)\leq \lambda_t(y)$.
\item If $x$ is self-adjoint, then $\lambda_t(f(x))=f(\lambda_t(x)), t\in (0, \tau(\mathbb{I}))$,
for every increasing continuous function $f$ on $\mathbb{R}$.
\end{enumerate}
\end{proposition}

\begin{example}
 Let $\mathcal{H}=\mathbb{C}^n$ and
 let $\mathcal{M}=\mathcal{B}(\mathcal{H})\cong \mathbb{M}_n(\mathbb{C})$
equipped with the normalized trace  $\tau_n:\triangleq \frac{1}{n}tr_n$ where
$tr_n$ is the standard trace on $\mathbb{M}_n(\mathbb{C})$.
  If $x\in \mathcal{B}(\mathcal{H})=\mathbb{M}_n(\mathbb{C})$ is self-adjoint,
then
$x$ can be written as $x=\sum_{i=1}^n\alpha_j p_j$,
where $\alpha_1\geq\alpha_2\geq\cdots\geq\alpha_n$ is the sequence
of eigenvalues of $x$ in which each is repeated according to its multiplicity
and $\sum_{i=1}^n p_j=\mathbb{I}$.
Therefore,
\[\lambda_t(x)=\sum_{j=1}^n\alpha_j\chi_{[\frac{j-1}{n}, \frac{j}{n})}, t\in [0, 1).\]
If $x\geq0$, then $\alpha_1\geq\alpha_2\geq\cdots\geq\alpha_n\geq0$,
$\lambda_t(x)=\mu_t(x)$
and $\mu_t^\ell(x)=\sum_{j=1}^n\alpha_j\chi_{(\frac{j-1}{n}, \frac{j}{n}]}$.

If $x\in \mathbb{M}_n(\mathbb{C})$ is arbitrary, then $\mu_t(x)=\mu_t(|x|)$ and the eigenvalues
of $|x|$ are usual called the singular values of $x.$
It follows that
\[\mu_t(x)=\sum_{j=1}^ns_j\chi_{[\frac{j-1}{n}, \frac{j}{n})}\]
and
\[\mu_t^\ell(x)=\sum_{j=1}^ns_j\chi_{(\frac{j-1}{n}, \frac{j}{n}]},\]
where $s_1\geq s_2\geq\cdots\geq s_n\geq0$ is the sequence of singular
values of $x$, repeated according to multiplicity. It is clear that
$\mu_{\frac{j-1}{n}}(x)=\mu_{\frac{j}{n}}^\ell(x)$.

Note that if $x\in \mathbb{M}_n(\mathbb{C})$ is self-adjoint, then
$x$ can also be  written as $x=\sum_{i=1}^m\beta_j p_j$, where
$\beta_1>\beta_2>\cdots>\beta_m(m\leq n)$.
Then
\[\lambda_t(x)=\sum_{j=1}^m\beta_j\chi_{[d_{j-1}, d_j)},\]
where $d_{j}=\sum_{i=1}^j\tau(p_i)$ for $j=1, 2, \cdots, m$
and $d_0=0.$ For each $j$, the length of the interval
$[nd_{j-1}, nd_j)$ is $n\tau_n(p_j)$, which is the dimension
of the eigenspace corresponding to $\beta_j.$
See \cite{H1997, DPS2019} for more details of
$\mu_t(\cdot)$ and $\lambda_t(\cdot)$ of operators and matrices
(Note:
the generalized singular values $\mu_{k}$, as defined in \cite{H1997},
is denoted by $\mu_{\frac{k}{n}}^\ell$, in this paper;
the generalized singular values $\mu_{\frac{k}{n}}$ and $\mu_{\frac{k}{n}}^\ell$
is noting but $\mu_{k+1}$ and $\mu_k$, respectively, in \cite{H1997})
\end{example}

\begin{example}
Consider the algebra $\mathcal{M}=L^\infty([0, 1])$ of all Lebesgue measurable essentially bounded
functions on $[0, 1]$. Algebra $\mathcal{M}$ can be seen as an abelian von Neumann algebra acting
via multiplication on the Hilbert space $\mathcal{H}=L^2([0, 1])$, with the trace given by integration
with respect to Lebesgue measure $m$.
For a real measurable function $f\in L^\infty([0, 1])$, the
decreasing rearrangement $f^*$ of the function $f$ is given by
\[f^*(t)=\inf\{s\in \mathbb{R}: m(\{h\in[0,1]: f(h)>s\})\leq t\}, 0<t<1.\]
Then $\mu_t(f)=|f|^*(t)$ and $\lambda_t(f)=f^*(t)$.
Suppose that $f=\sum_1^n\alpha_i \chi_{B_i}$, where $B_i\subseteq[0, 1]$ with $B_i\cap B_j=\emptyset$
whenever $i\neq j$, and $0<\alpha_j\in \mathbb{R} (j = 1,2,\cdots, n)$ are such that $\alpha_i\neq \alpha_j$
whenever $i\neq j$. For the computation of $\mu_t(f)$, it may be assumed that
$\alpha_1>\alpha_2>\cdots>\alpha_n$. Then
\[\lambda_t(f)=\sum_{j=1}^n\alpha_j\chi_{[d_{j-1}, d_j)},\]
where $d_{j}=\sum_{i=1}^jm(B_i)$ for $j=1, 2, \cdots, n$
and $d_0=0.$
If $f\geq0$, then $\alpha_1>\alpha_2>\cdots>\alpha_n\geq0$,
$\lambda_t(f)=\mu_t(f)$
and $\mu_t^\ell(f)=\sum_{j=1}^n\alpha_j\chi_{(d_{j-1}, d_j]}$.
See \cite{DPS2019, N1987} for more details.
\end{example}

\subsection{Fuglede-\,-Kadison determinant}
Let $\mathcal{M}$  be a finite von
Neumann algebra acting on a separable Hilbert space $\mathcal{H}$,
with a normal faithful finite tracial state $\tau$.
Recall that the Fuglede-\,-Kadison determinant
$\Delta=\Delta_\tau:\mathcal{M}\rightarrow \mathbb{R}^+$ is defined by
$\Delta_\tau(x)=\tau(\log|x|)$ if $|x|$ is invertible;
and otherwise, we define $\Delta_\tau(x)=\inf\Delta_\tau(|x| + \varepsilon \mathbb{I})$,
the infimum takes over all scalars $\varepsilon>0$.
We define Fuglede-\,-Kadison determinant-like function of $x$
by \[\Lambda_t(x)=\exp\{\int_0^t\log\mu_s(x)ds\}, t>0.\]
Since $\tau(\mathbb{I})=1$,  if $|x|$ is invertible, then
\[\Delta_\tau(x)=\Lambda_1(x)=\exp\{\int_0^1\log\mu_s(x)ds\}.\]
We understanding that $\Delta(x)=0$ if
\[\int_0^{\tau(\mathbb{I})}\log\mu_s(x)ds=-\infty.\]
Recall that $x$ is said to be logarithmically
submajorised by $y$(see \cite{DDSZ2020, HSZ2020}), denoted by $x\prec\prec_{\log}y$
(or $\mu(x)\prec\prec_{\log}\mu(y)$), if $\Lambda_t(x)\leq\Lambda_t(y)$ for all $t>0$.

We state for easy reference the following fact,
obtained from \cite{A1967, B1983} for Fuglede-Kadison determinant  which will be applied below.
\begin{proposition}\label{proposition 2.4}
Let $x, y\in\mathcal{M}$. Then
\begin{enumerate}
\item
$\Delta_\tau(\mathbb{I})=1, \Delta_\tau(xy)=\Delta_\tau(x)\Delta_\tau(y),$
\item $\Delta_\tau(x)=\Delta_\tau(x^*)=\Delta_\tau(|x|),~\Delta_\tau(|x|^\alpha)
=(\Delta_\tau(|x|))^\alpha, \alpha\in \mathbb{R}^+$
\item
$\Delta_\tau(x^{-1})=(\Delta_\tau(x))^{-1}, \mbox{if}~ x~\mbox{is invertible in}~\mathcal{M}$
\item $\Delta_\tau(x)\leq\Delta_\tau(y), \mbox{if}~ 0\leq x\leq y$
\item $\lim_{\varepsilon\rightarrow0^+}
\Delta_\tau(x+\varepsilon1)=\Delta_\tau(x), \mbox{if}~   0\leq x.$
\item $\Delta_\tau(x)\leq\Delta_\tau(y), \mbox{if}~  x\prec\prec_{\log} y$.
\end{enumerate}
\end{proposition}
  See \cite{A1967, B1983,  BL2008} for basic properties and detailed information
on Fuglede-Kadison determinant of $x\in\mathcal{M}$.

\begin{example}
 Let $\mathcal{H}=\mathbb{C}^n$ and let $\mathcal{M}=\mathcal{B}(\mathcal{H})\cong \mathbb{M}_n(\mathbb{C})$
equipped with the normalized trace  $\tau_n:\triangleq \frac{1}{n}tr_n$ where
$tr_n$ is the standard trace on $\mathbb{M}_n(\mathbb{C})$.
  If $x\in \mathcal{B}(\mathcal{H})$, then
$\Delta_{\tau_n}(x)=(\det(|x|))^{\frac{1}{n}}$.
See \cite{H1997} for more information
on determinant of matrices.
\end{example}

If $x, y\in\mathcal{M}$ and $0<p<\infty$,
then $x$ is said to be $p$-submajorised by $y$,
denoted by $x\prec\prec_p y$, if
$\int_0^t\mu_s(x)^pds\leq \int_0^t\mu_s(y)^pds$
for all $t>0.$
\begin{remark}\label{rk:determanint 2}
Let $x, y\in\mathcal{M}^+$ be invertible.
Then the following conditions are equivalent:
\begin{enumerate}
\item $\mathbb{I}+rx\prec\prec_{\log} \mathbb{I}+ry,$ for all $ r\in \mathbb{R}^+$;
\item $x\prec\prec_{p} y,~~0<p<1$;
\item $x\prec\prec_{\log} y$;
\item $\int_0^t\varphi(\mu_s(x))ds\leq\int_0^t\varphi(\mu_s(y))ds$ for all $t>0$ and all nondecreasing functions $\varphi$
on $[0, \infty)$ such that $\varphi(0)=0$ and $t\rightarrow \varphi(e^t)$ is convex.
\end{enumerate}
Indeed,
let $\psi$ is a bounded positive measurable function on $[0, \infty)$ and $\pi_t(r)=\exp\{\int_0^t\log(1+r\psi(s))ds\}$.
By \cite[Lemma 3.2]{F1983}, we have
\[
\int_0^t\psi(s)^pds=\frac{p\sin(\pi p)}{\pi}\int_0^\infty \frac{\log\pi_t(r)}{r^{p+1}}dr,
\]
which implies that (1)$\Rightarrow$(2) holds.

Note that if $(\int_0^{t}|\varphi(s)|^p\frac{ds}{t})^{\frac{1}{p}}<\infty,t>0$ for some $p>0$,
then from \cite[p.71]{R1974} we obtain
\[\exp\{\int_0^{t}\log|\varphi(s)|\frac{ds}{t}\}=\lim_{p\rightarrow0}
(\int_0^{t}|\varphi(s)|^p\frac{ds}{t})^{\frac{1}{p}}, t>0,\]
which yields (2)$\Rightarrow $(3). (3)$\Rightarrow$(4) follows from the fact that $t\rightarrow \varphi(e^t)$ is convex and
$\varphi(e^{\log\mu(x)})=\varphi(\mu(x))$(see \cite[p.22, Theorem D.2]{MO1979}).
 It is easy to check that (4)$\Rightarrow$(1).
\end{remark}

\section{Unitary approximation and Logarithmic submajorisation}

Our starting point is the following inequality for complex numbers:
\begin{equation}\label{ineq 3.1}
||z|-1|\leq||z|-v|\leq||z|+1|, z, v\in \mathbb{C}~\mbox{with}~|v|=1.
\end{equation}
In this section, we will consider some Logarithmic submajorisation
inequalities for operator version of (\ref{ineq 3.1}).
We start with a lemma which will be used in our proof.

\begin{lemma}\label{lemma 0000}
Let $x\in \mathcal{M}^+$. Then
\[\lambda_s(-x)=-\mu_{1-s}^\ell(x), 0<s<1.\]
\end{lemma}
\begin{proof}
Let $x=\sum_{i=1}^n\alpha_i p_i$ with $\alpha_1>\alpha_2>\cdots>\alpha_n\geq0$ and
$ p_ip_j=0, i\neq j$.
Without loss of generality we can assume
$\sum_{i=1}^n p_i= \mathbb{I}$. Indeed, if
$\sum_{i=1}^n p_i\neq\mathbb{I}$ we write
$p_{n+1}=\mathbb{I}-\sum_{i=1}^n p_i$. Replacing $p_n$ by $p_n+p_{n+1}$ in the equation
$x=\sum_{i=1}^n\alpha_i p_i$ if
$\alpha_n=0$, and replacing $x=\sum_{i=1}^n\alpha_i p_i$ by $x=\sum_{i=1}^{n+1}\alpha_i p_i$ if
$\alpha_n\neq0$.
Set $d_i=\sum_{j=1}^i\tau(p_j), 1\leq i\leq n$ and $d_0=0$.
Then $d_n=\sum_{j=1}^n\tau(p_j)=\tau(\mathbb{I})=1$,
\[\mu_s^\ell(x)=\sum_{i=1}^n\alpha_i\chi_{(d_{i-1}, d_i]}(s), 0<s<1,\]
and
\[\lambda_s(-x)=\sum_{j=1}^n-\alpha_{n-j+1}\chi_{[1-d_j, 1-d_{j-1})}, 0<s<1.\]
Thus,
\[\mu_{1-s}^\ell(x)=-\lambda_s(-x), 0<s<1,\]
For the general case, let $0\leq x\in \mathcal{M}$ and let
$x=\int_0^{\|x\|} \lambda de_\lambda(x)$ be the spectral decomposition of $x$.
Put \[f_k(t)=\sum_{j=1}^{2^n} \frac{j\|x\|}{2^n}
\chi_{[\frac{(j-1)\|x\|}{2^n},  \frac{j\|x\|}{2^n})}.\]
Write
$x_n=f_n(|x|)$.
It follows that $\|x-x_n\|\leq\frac{\|x\|}{2^n}$ and $x_n\geq x_{n+1}\geq x$.
The proof is completed by showing that
\[\lim_{n\rightarrow\infty}\mu^\ell_s(x_n)=\mu_{s}^\ell(x)~
 \mbox{and}~
\lim_{n\rightarrow\infty}\lambda_s(-x_n)=\lambda_{s}(-x).\]
For $\epsilon>0,$ we obtain
\[\mu^\ell_s(x_n)\leq\mu_{s-\epsilon}^\ell(x)+\mu_\epsilon^\ell(x_n-x)
\leq\mu_{s-\epsilon}^\ell(x)+\|x_n-x\|, 0<s<1\]
and
\[\lambda_{s+\epsilon}(-x)-\|x_n-x\|\leq
\lambda_{s+\epsilon}(-x)-\lambda_\epsilon(x_n-x)\leq\lambda_{s}(-x_n), 0<s<1.
\]
Taking the $n\rightarrow\infty$ of the both side, we get
\[\limsup_{n\rightarrow\infty}\mu^\ell_s(x_n)\leq\mu_{s-\epsilon}^\ell(x), 0<s<1\]
and
\[\liminf_{n\rightarrow\infty}\lambda_s(-x_n)\geq\lambda_{s+\epsilon}(x), 0<s<1.\]
Since $\lambda(-x)$
is right-continuous and $\mu^\ell(x)$
is left-continuous on (0, 1),
letting $\epsilon\downarrow0$, we obtain
\[\limsup_{n\rightarrow\infty}\mu^\ell_s(x_n)\leq\mu_{s}^\ell(x),~
~\liminf_{n\rightarrow\infty}\lambda_s(-x_n)\geq\lambda_{s}(-x), 0<s<1.\]
On the other hand, since $x\leq x_n$, $-x\geq-x_n$, moreover,
\[\liminf_{n\rightarrow\infty}\mu^\ell_s(x_n)\geq\mu_{s}^\ell(x),~
\limsup_{n\rightarrow\infty}\lambda_s(-x_n)\leq\lambda_s(-x), 0<s<1.\]
Hence $\lim_{n\rightarrow\infty}\mu^\ell_s(x_n)=\mu_{s}^\ell(x),~
\limsup_{n\rightarrow\infty}\lambda_s(-x_n)=\lambda_s(-x), 0<s<1$.
This completes the proof.
\end{proof}

\begin{proposition}\label{lemma 3.1}
Let $x\in \mathcal{M}$. Then
\[-\mu_{1-s}^\ell(x)\leq\lambda_s(Rex)\leq\mu_s(x), ~~-\mu_{1-s}^\ell(x)\leq\lambda_s(Imx)\leq\mu_s(x), 0<s<1.\]
\end{proposition}
\begin{proof}
The proof is adapted from \cite[Lemma 2.1]{JL2020}.
For any $t>0,$ we have
\[t^2x^*x+\frac{1}{t^2}\mathbb{I}-(x^*+x)=
(tx-\frac{1}{t}\mathbb{I})^*(tx-\frac{1}{t}\mathbb{I})\geq0.\]
which tell us that
\begin{equation}\label{inequa. lemma 3.1}
t^2x^*x+\frac{1}{t^2}\mathbb{I}\geq 2Rex\geq-(t^2x^*x+\frac{1}{t^2}\mathbb{I}).
\end{equation}
Proposition \ref{proposition 2.2}(8) now yields
\begin{equation}\label{inequ. lemma 3.1(2)}
\lambda_s(t^2x^*x+\frac{1}{t^2}\mathbb{I})\geq
\lambda_s(2Rex)\geq\lambda_s(-(t^2x^*x+\frac{1}{t^2}\mathbb{I})), 0<s<1
\end{equation}
for all $s>0.$
An easy calculation shows that
$\mu_s(y+\mathbb{I})=\mu_s(y)+1$ and
$\mu_s^\ell(y+\mathbb{I})=\mu_s^\ell(y)+1$ for $y\geq0$ and $0<s<1.$
Combining this with (\ref{inequ. lemma 3.1(2)})
and Proposition \ref{proposition 2.2} we can assert that
\begin{align*}t^2\mu_s(x)^2+\frac{1}{t^2}1&=
\mu_s(t^2x^*x+\frac{1}{t^2}\mathbb{I})\\
&=\lambda_s(t^2x^*x+\frac{1}{t^2}\mathbb{I})\\
&\geq
\lambda_s(2Rex), 0<s<1.
\end{align*}
If it was true that $\mu_s(x)=0$ for some $s>0$,
there would be $2\lambda_{t}(Rex)=\lambda_{t}(2Rex)\leq0$
by take $t\rightarrow\infty$.
Otherwise,
we take $t=\frac{1}{\mu_s(x)^{\frac{1}{2}}}$, it follows that
$\lambda_s(Rex)\leq\mu_s(x), 0<t<1.$

On the other hand, combining (\ref{inequ. lemma 3.1(2)}) with
 Proposition \ref{proposition 2.2} and
 Lemma \ref{lemma 0000} yields
\begin{align*}
\lambda_s(2Rex)&\geq \lambda_s(-(t^2x^*x+\frac{1}{t^2}\mathbb{I}))\\
&=-\mu_{1-s}^\ell(t^2x^*x+\frac{1}{t^2}\mathbb{I})\\
&=-t^2\mu_{1-s}^\ell(x^*x)-\frac{1}{t^2}, 0<s<1.
\end{align*}
We now apply the above argument again, with $\mu(x)$ replaced by $\mu^\ell(x)$, to obtain
$\lambda_s(Rex)\geq-\mu_{1-s}^\ell(x), 0<s<1.$
Finally, since $Re(-ix)=Im x$, from what has already been proved  we see that
$-\mu_{1-s}^\ell(x)=-\mu_{1-s}^\ell(-ix)\leq\lambda_s(Imx)\leq\mu_s(-ix)=\mu_s(x), 0<s<1.$
\end{proof}

\begin{remark}\label{rk:lemma 3.1}
\begin{enumerate}
\item Let $x\in \mathcal{M}$. From the proof of Lemma \ref{lemma 0000} we have
\[\lambda_s(Rex)\leq\mu_s(x), ~~\lambda_s(Imx)\leq\mu_s(x), 0<s<1.\]

\item Let $x\in \mathcal{M}$.
It follows from inequality (\ref{inequa. lemma 3.1}) that
\[t^2x^*x+\frac{1}{t^2}\mathbb{I}\geq 2Rex\geq -(t^2x^*x+\frac{1}{t^2}\mathbb{I}).\]
Moreover, \cite[Lemma 4.2]{HSZ2020} means that
$\mu(2Rex)\prec\prec_{\log}\mu(t^2x^*x+\frac{1}{t^2}\mathbb{I})$ for all $t>0,$
with $-\infty$ allowed for values.
Moreover, we have
\[\Delta_\tau(2Rex)\leq\Delta_\tau(t^2x^*x+\frac{1}{t^2}\mathbb{I}), t>0.\]
\end{enumerate}
\end{remark}

\begin{corollary}\label{corollary 3.2}
Let $x, y\in \mathcal{M}$ and let $\alpha\in \mathbb{R}$. If $x^{*}=x$, then
\[ \lambda_s(iRey-iy)\leq\mu_s(y-\alpha x), 0<s<1\]
and
\[ \lambda_s(y-iIm y)\leq\mu_s(y-i\alpha x), 0<s<1.\]
\end{corollary}
\begin{proof} The results follow from Remark \ref{rk:lemma 3.1}(1)
along with the fact that  $-i(y-Rey)=Imy=Im(y-\alpha x)$ and
$y-i Imy=Rey=Re(y-i\alpha x)$.
\end{proof}

\begin{proposition}\label{proposition 3.5}
Let $0\leq x\in \mathcal{M}$ such that $\|x\|>1$.
\begin{enumerate}
\item If $u\in \mathcal{M}$ is an unitary operator, then
\[\mu(x-Reu)\prec\prec_{\log}\mu(x+\mathbb{I}),\]
which implies that
\[\Delta_\tau(x-Reu)\leq\Delta_\tau(x+\mathbb{I}).\]
\item If $u\in \mathcal{M}$ is an unitary operator and
$\tau(|x-\mathbb{I}|)=\tau(|x-u|)$, then
\[\Delta_\tau(x-u)\leq\Delta_\tau(x-\mathbb{I}).\]
\end{enumerate}
\end{proposition}
\begin{proof}
(1). From $-\mathbb{I}\leq -Reu\leq\mathbb{I}$, we deduce that
$-(x+\mathbb{I})\leq x-\mathbb{I}\leq x-Reu\leq x+\mathbb{I}.$
Then we conclude from \cite[Lemma 4.2]{HSZ2020} that
\[\mu(x-Reu)\prec\prec_{\log}\mu(x+\mathbb{I}).\]
Hence we see that
\[\Delta_\tau(x-Reu)\leq\Delta_\tau(x+\mathbb{I}).\]

(2). Note that
\cite[Corollary 2.6]{DD1992} leads to
\[\mu(x-\mathbb{I})\prec\prec\mu(x-u).\]
Since $\tau(|x-\mathbb{I}|)=\tau(|x-u|)$,  \cite[Theorem 3.3]{BR2014} shows
that
$\tau(|x-u|^p)\leq\tau(|x-\mathbb{I}|^p), 0<p<1$, i.e.
\[\int_0^{1}\mu_t(x-u)^p\leq
\int_0^{1}\mu_t(x-\mathbb{I})^pdt, 0<p<1.\]
Hence, from
$\int_0^{1}|f(s)|^pds)^{\frac{1}{p}}<\infty$
and \cite[p.74]{R1974} we obtain
\[exp\{\int_0^{1}\log|f(s)|ds\}=\lim_{p\rightarrow0}
(\int_0^{1}|f(s)|^pds)^{\frac{1}{p}},\]
which force
$\Delta_\tau(x-u)\leq\Delta_\tau(x-\mathbb{I}).$
\end{proof}

\begin{lemma}\label{lemma3.6}
Let $0\leq x\in \mathcal{M}$ be invertible. Then
\[\mu_t^\ell(x^{-1})=\mu_{1-t}(x)^{-1}, 0<t<1.\]
\end{lemma}
\begin{proof}
Without loss of generality, we may assume that
$\mathcal{M}$ has no minimal projections
(otherwise we consider the von Neumann algebra
$\mathcal{M}\otimes L^{\infty}([0, 1])$).
First we assume that
$x=\sum_{i=1}^n\alpha_i p_i$ with $\alpha_1>\alpha_2>\cdots>\alpha_n>0$ and
$\sum_{i=1}^n p_i=1, p_ip_j=0, i\neq j$.
Thus $x^{-1}=\sum_{i=1}^n\frac{1}{\alpha_i} p_i$.
Let $d_i=\sum_{j=1}^i\tau(p_j), 1\leq i\leq n$. Then $d_n=\tau(\mathbb{I})=1$,
\[\mu_t(x)=\alpha_1\chi_{(0, d_1)}+\sum_{i=2}^n\alpha_i\chi_{[d_{i-1}, d_i)}, 0<t<1,\]
and
\[\mu_t^\ell(x^{-1})=
\sum_{i=2}^n\frac{1}{\alpha_i}\chi_{(1-d_{i}, 1-d_{i-1}]}+\frac{1}{\alpha_1}\chi_{(1-d_1, 1)}, 0<t<1.\]
Therefore,
\[\mu_t^\ell(x^{-1})=\frac{1}{\mu_{1-t}(x)}, 0<t<1.\]
For the general case, let $0\leq x\in \mathcal{M}$. Since $x^{-1}\in \mathcal{M}$,
there exists $\delta>0$ such that
$x=\int_\delta^{\|x\|} \lambda de_\lambda(x)$ is the spectral decomposition of $x$.
Put \[f_k(t)=\sum_{j=1}^{2^n}(\delta+\frac{(j-1)a}{2^n})
\chi_{[\delta+\frac{(j-1)a}{2^n}, \delta+\frac{ja}{2^n})},\]
 where $a=\|x\|-\delta>0$. Obviously,
$0\leq f_k(t)\leq f_{k+1}(t)\leq t.$
Set
\[
x_n=f_n(|x|)=\sum_{j=1}^{2^n}(\delta+\frac{(j-1)a}{2^n})
e_{[\delta+\frac{(j-1)a}{2^n}, \delta+\frac{ja}{2^n})}(x).
\]
Then
\[
x_n^{-1}=f_n(|x|)=\sum_{j=1}^{2^n}(\delta+\frac{(j-1)a}{2^n})^{-1}
e_{[\delta+\frac{(j-1)a}{2^n}, \delta+\frac{ja}{2^n})}(x).
\]
It follows that $\|x-x_n\|\leq\frac{a}{2^n}$ and
\[\|x^{-1}-x_n^{-1}\|\leq (\delta+\frac{(j-1)a}{2^n})^{-1}-
(\delta+\frac{ja}{2^n})^{-1}\leq\frac{1}{\delta^2}\frac{a}{2^{n}}.\]
Hence we infer from \cite[Lemma 3.4]{FK1986} that
$\mu_t(x)=\lim_{n\rightarrow\infty}\mu_t(x_n)$.
On the other hand, picking up a small $\epsilon>0,$ we obtain
\[\mu^\ell_t(x^{-1}_n)\leq\mu_{t-\epsilon}^\ell(x^{-1})+\mu_\epsilon^\ell(x^{-1}-x_n^{-1})
\leq\mu_{t-\epsilon}^\ell(x^{-1})+\|x^{-1}-x_n^{-1}\|.\]
Letting $\epsilon\downarrow0$ we get
\[\mu^\ell_t(x^{-1}_n)
\leq\mu_{t}^\ell(x^{-1})+\|x^{-1}-x_n^{-1}\|.\]
In consequence,  $\limsup_{n\rightarrow\infty}\mu^\ell_t(x^{-1}_n)\leq\mu_{t}^\ell(x^{-1})$.
Therefore,  $x^{-1}\leq x_n^{-1}$ tells us that
\[\liminf_{n\rightarrow\infty}\mu^\ell_t(x^{-1}_n)\geq\mu_{t}^\ell(x^{-1}).\]
Hence
$\mu_t^{\ell}(x)=\lim_{n\rightarrow\infty}\mu_t^\ell(x_n)$.
This completes the proof.
\end{proof}

\begin{example}
 Let $\mathcal{H}=\mathbb{C}^n$ and let $\mathcal{M}=\mathcal{B}(\mathcal{H})\cong \mathbb{M}_n(\mathbb{C})$
equipped with the normalized trace  $\tau_n:\triangleq \frac{1}{n}tr_n$ where
$tr_n$ is the standard trace on $\mathbb{M}_n(\mathbb{C})$.
  If $x\in \mathbb{M}_n(\mathbb{C})$ is positive and invertible, then
$x$ can be written as $x=\sum_{i=1}^n\alpha_j p_j$,
where $\alpha_1\geq\alpha_2\geq\cdots\geq\alpha_n>0$ is the sequence
of eigenvalues of $x$ in which each is repeated according to its multiplicity
and $\sum_{i=1}^n p_j=\mathbb{I}$.
The proof of Lemma \ref{lemma3.6} tells us that
\[\mu_t^\ell(x^{-1})=\mu_{\tau(\mathbb{I})-t}(x)^{-1}, 0<t<\tau(\mathbb{I}),\]
i.e.
\[\mu_{\frac{k}{n}}^\ell(x^{-1})=(\alpha_{n+1-k})^{-1}=\mu_{\tau(\mathbb{I})-\frac{k}{n}}(x)^{-1}.\]
\end{example}

We conclude this section with a series of inequalities of generalized singular value function.
\begin{lemma}\label{lemma 3.7}
Let $x, y\in\mathcal{M}$.
\begin{enumerate}
\item If $x^*=x$, then $\lambda_t(x)\leq\mu_t(x).$
\item If $s, t>0$ such that $s+t<1$, then
$1\leq\mu_t(x)+\mu_s(\mathbb{I}-x)$ and $1\leq\mu_t^\ell(x)+\mu_s^\ell(\mathbb{I}-x)$.
\item For any $t>0$ we have
$1\leq\mu_t(x)+\mu_{1-t}^\ell(\mathbb{I}-x)$, $1\leq\mu_t^\ell(x)+\mu_{1-t}^\ell(\mathbb{I}-x)$
and $1\leq\mu_t^\ell(x)+\mu_{1-t}(\mathbb{I}-x)$.
\item For any $t>0$ we have
$1\leq\mu_t(x)+\mu_{1-t}^\ell(x\pm i\mathbb{I})$,
$1\leq\mu_t^\ell(x)+\mu_{1-t}^\ell(x\pm i\mathbb{I})$
and $1\leq\mu_t^\ell(x)+\mu_{1-t}(x\pm i\mathbb{I})$.
\item If $0\leq x\in \mathcal{M}$ and $\|x\|\leq1$, then
\[
\mu_t(1-x)=1-\mu_{1-t}^\ell(x),~~\mu_t^\ell(1-x)=1-\mu_{1-t}(x).
\]
\end{enumerate}
\end{lemma}
\begin{proof}
(1). Since $-|x|\leq x\leq |x|$, $\lambda_t(x)\leq\lambda_t(|x|)=\mu_x(x).$  (2)-(4)
follow from the fact $\mu_{s+t}(x+y)\leq\mu_t(x)+\mu_s(y)$ and
$\mu_{s+t}^\ell(x+y)\leq\mu_t^{\ell}(x)+\mu_s^\ell(y)$.
 (5). This follows by the same method as in Lemma \ref{lemma3.6}.
\end{proof}

\section{Harnack type inequality for operator}
In this section Harnack type inequalities for operators
in Logarithmic submajorisation are stated and proved.
We will extend the results of Yang-\,-Zhang\cite{YZ2020} and Lin-\,-Zhang\cite{LZ2017}
to the case of finite von Neumann algebra.
We start with a lemma which follows by the same method as in
\cite[Proposition 2]{YZ2020}.
\begin{lemma}\label{lemma 3.8}
Let $x\in \mathcal{M}$. If  $\mathbb{I}-x$ is invertible, then
\begin{align*}
(\mathbb{I}-x^*)^{-1}(\mathbb{I}-x^*x)(\mathbb{I}-x)^{-1}
=&2Re((\mathbb{I}-x)^{-1})-\mathbb{I}\\
=&2Re((\mathbb{I}-x)^{-1}-\frac{1}{2}\mathbb{I})\\
=&Re((\mathbb{I}+x)(\mathbb{I}-x)^{-1})=S^*S,
\end{align*}
where $S=(\mathbb{I}-x^*x)^{\frac{1}{2}}(\mathbb{I}-x)^{-1}$.
Moreover, if $x\in\mathcal{M}$ with $\|x\|<1$, then $\mathbb{I}-x$ is invertible,
which implies that the equalities above are true.
\end{lemma}

\begin{theorem}\label{theorem 3.9}
Let $x\in \mathcal{M}$ with $\|x\|<1$. Then
\begin{equation}\label{inequ. theorem 3.9}
\mu_t((\mathbb{I}-x^*)^{-1}(\mathbb{I}-x^*x)(\mathbb{I}-x)^{-1})
\leq\frac{1+\mu_t(x)}{1-\mu_t(x)}, 0<t<1.
\end{equation}
Moreover, for any subset $K\subseteq[0, 1]$ we have
\begin{align*}
\int_K\log\mu_t((\mathbb{I}-x^*)^{-1}(\mathbb{I}-x^*x)(\mathbb{I}-x)^{-1})dt
&\leq\int_K\log\frac{1+\mu_t(x)}{1-\mu_t(x)}dt\\
&\leq\int_0^1\log\frac{1+\mu_t(x)}{1-\mu_t(x)}dt.
\end{align*}
In particular,
\begin{align*}
\frac{\Delta_\tau(\mathbb{I}-x^*x)}{\Delta_\tau(\mathbb{I}-x)^2}
\leq\exp\int_0^1\log\frac{1+\mu_t(x)}{1-\mu_t(x)}dt.
\end{align*}
\end{theorem}
\begin{proof}
We conclude from the definition of $\mu_t(\cdot)$ and $\lambda_t(\cdot)$ that
\begin{align*}
\mu_t((\mathbb{I}-x^*)^{-1}(\mathbb{I}-x^*x)(\mathbb{I}-x)^{-1})
=&\lambda_t((\mathbb{I}-x^*)^{-1}(\mathbb{I}-x^*x)(\mathbb{I}-x)^{-1})\\
=&\lambda_t(2Re((\mathbb{I}-x)^{-1})-\mathbb{I})~(Lemma~\ref{lemma 3.8})\\
=&\lambda_t(2Re((\mathbb{I}-x)^{-1}))-1~(Proposition~\ref{proposition 2.2}(9))\\
\leq&\mu_t(2 (\mathbb{I}-x)^{-1})-1~(Remark~\ref{rk:lemma 3.1})\\
=&\frac{2}{\mu_{1-s}^\ell(\mathbb{I}-x)}-1~(Lemma~\ref{lemma3.6})\\
\leq&\frac{2}{1-\mu_t(x)}-1~(Lemma~\ref{lemma 3.7})\\
=&\frac{1+\mu_t(x)}{1-\mu_t(x)}, 0<t<1.
\end{align*}
Furthermore, since $\frac{1+\mu_t(x)}{1-\mu_t(x)}\geq1$, (\ref{inequ. theorem 3.9}) means that
\begin{align*}
\int_K\log\mu_t((\mathbb{I}-x^*)^{-1}(\mathbb{I}-x^*x)(\mathbb{I}-x)^{-1})dt
&\leq\int_K\log\frac{1+\mu_t(x)}{1-\mu_t(x)}dt\\
&\leq\int_0^1\log\frac{1+\mu_t(x)}{1-\mu_t(x)}dt.
\end{align*}
Finally, by (\ref{inequ. theorem 3.9}) and Proposition \ref{proposition 2.4}(1)-(3), we have
\begin{align*}
\frac{\Delta_\tau(\mathbb{I}-x^*x)}{\Delta_\tau(\mathbb{I}-x)^2}
&=\Delta_\tau((\mathbb{I}-x^*)^{-1}(\mathbb{I}-x^*x)(\mathbb{I}-x)^{-1})\\
&=\exp\int_0^1\log\mu_t((\mathbb{I}-x^*)^{-1}(\mathbb{I}-x^*x)(\mathbb{I}-x)^{-1})dt\\
&\leq\exp\int_0^1\log\frac{1+\mu_t(x)}{1-\mu_t(x)}dt.
\end{align*}
\end{proof}

To achieve one of our main results, we state for easy reference the following fact,
 which will be applied below.
\begin{lemma}\label{lemma 4.1}
Let $x, y\in\mathcal{M}$ be invertible.
If $K$ is a Borel subset of $[0, 1]$ with $m(K)=t$ $(m(K)$ denotes the
Lebesgue measure of $K)$, then
\begin{align*}
\int_K\log\mu_s(x)ds+\int_{0}^t\log\mu_{1-s}(y)ds
\leq
\int_K\log\mu_s(xy)ds.
\end{align*}
\end{lemma}
\begin{proof}
Let $K^c$ denote the set $\{t\in [0, 1]: t\notin E\}$.
Then $m(K^c)=1-t$.
We conclude from \cite[Theorem 2]{N1987} that
\begin{equation}\label{equat. lemma 4.1(1)}
\int_{K^c}\log\mu_s(xy)ds\leq\int_{K^c}\log\mu_s(x)+\int_0^{1-t}\log\mu_s(y)ds.
\end{equation}
Note that $x, y\in\mathcal{M}$ are invertible.
By Proposition \ref{proposition 2.4}(1) and (3) we have $\Delta(x)\neq0, \Delta(y)\neq0$ and
\begin{equation}\label{equat. lemma 4.1(2)}
-\infty<\int_0^1\log(\mu_s(x))ds+\int_{0}^1\log\mu_{s}(y)ds=\int_{0}^1\log\mu_{s}(xy)ds<\infty
\end{equation}
Subtracting (\ref{equat. lemma 4.1(1)}) from (\ref{equat. lemma 4.1(2)}) yields
\[\int_E\log\mu_s(x)ds+\int_{1-t}^1\log\mu_{s}(y)ds\leq
\int_E\log\mu_s(xy)ds,\]
i.e.,
\[\int_E\log\mu_s(x)ds+\int_{0}^t\log\mu_{1-s}(y)ds\leq
\int_E\log\mu_s(xy)ds.\]
\end{proof}

\begin{remark}\label{remark 4.2}
\begin{enumerate}
\item Let $x, y\in\mathcal{M}$ and let
  $K$ be a Borel subset of $[0, 1]$ with $m(K)=t$ (here $m(K)$ denotes the
Lebesgue measure of $K)$. Then
\begin{align*}
\int_K\log\mu_s(x)ds+\int_{0}^t\log\mu_{1-s}(y)ds
\leq
\int_K\log\mu_s(xy)ds.
\end{align*}
Indeed, if $x, y$ are invertible, then it follows from Lemma \ref{lemma 4.1}.
We write $x=u|x|$ and $y=v|y|$ for unitary operators $u, v\in \mathcal{M}$.
Then $z=u|x||y^*|v^*$ and $\mu_t(x)=\mu_t(|x|)$, $\mu_t(y)=\mu_t(|y^*|)$,
$\mu_t(z)=\mu_t(|x||y^*|)$. Thus, we may without loss of generality assume
$x\geq0, y\geq0$ and let
\[z(\epsilon_1, \epsilon_2)=(x+\epsilon_1\mathbb{I})(y+\epsilon_2\mathbb{I}).\]
Note that $\mu_s(x+\epsilon_1\mathbb{I})=\mu_s(x)+\epsilon_1$ and
$\mu_s(y+\epsilon_2\mathbb{I})=\mu_s(y)+\epsilon_2$. From Lemma \ref{lemma 4.1}
we see that
\begin{equation}\label{inequ. remark 4.2}
\begin{split}
&\int_K\log(\mu_s(x)+\epsilon_1)ds+\int_{0}^t\log(\mu_{1-s}(y)+\epsilon_2)ds\\
\leq&
\int_K\log\mu_s(z(\epsilon_1, \epsilon_2))ds.
\end{split}
\end{equation}
Moreover, for any projection operators $e\in \mathcal{M}$, we have
\[\|z(\epsilon_1, \epsilon_2)e\|^2=\|e(y+\epsilon_2\mathbb{I})
(x^2+2\epsilon_1x+\epsilon_1^2\mathbb{I})
(y+\epsilon_2\mathbb{I})e\|,\]
which implies that
$\mu_s(z(\epsilon_1, \epsilon_2))$ is decreasing in $\epsilon_1$.
Similarly,
$\mu_s(z(\epsilon_1, \epsilon_2))$ is decreasing in $\epsilon_2$.
Letting $\epsilon_i\rightarrow0$ and using the monotone convergence
theorem in (\ref{inequ. remark 4.2}), we obtain the desired inequality.
\item Let $x, y\in\mathcal{M}$ and let
 $K$ be a Borel subset of $[0, 1]$ with $m(K)=t$ $(m(K)$ denotes the
Lebesgue measure of $K)$. Combing \cite[Theorem 2]{N1987} with Lemma \ref{lemma 4.1} we can assert that
\begin{align*}
\int_K\log\mu_s(x)ds+\int_{0}^t\log\mu_{1-s}(y)ds
&\leq
\int_K\log\mu_s(xy)ds\\
&\leq\int_K\log\mu_s(x)+\int_0^t\log\mu_s(y)ds.
\end{align*}
In particular, if $K=[0, t]$, then
\begin{align*}
\int_0^t\log\mu_s(x)ds+\int_{0}^t\log\mu_{1-s}(y)ds
&\leq
\int_0^t\log\mu_s(xy)ds\\
&\leq\int_0^t\log\mu_s(x)+\int_0^t\log\mu_s(y)ds.
\end{align*}
\end{enumerate}
\end{remark}

\begin{theorem}\label{theorem 3.10}
Let $x\in \mathcal{M}$ with $\|x\|<1$.
If $K$ is a Borel subset of $[0, 1]$ with $m(K)=t$ $(m(K)$ denotes the
Lebesgue measure of $K)$, then
\begin{align*}
&\int_K\log\mu_{s}((\mathbb{I}-x^*)^{-1}(\mathbb{I}-x^*x)(\mathbb{I}-x)^{-1})ds\\
&\geq
\int_{0}^{t}2\log\frac{1}{1+\mu_{s}(x)}ds+\int_K\log(1-\mu_{1-s}(x)^2)ds, t>0.
\end{align*}
\end{theorem}
\begin{proof}
For convenience, we write $A:=(\mathbb{I}-x^*)^{-1}(\mathbb{I}-x^*x)(\mathbb{I}-x)^{-1}$.
Since $\|x\|<1$, $A$ is invertible, hence that $\Delta(A)>0.$
Therefore, $\int_0^{1}\log\mu_{s}(A)ds>-\infty$.
Using Lemma \ref{lemma 4.1} twice, we have
\[\int_{K}\log\mu_{s}(A)ds\geq
\int_{0}^{t}2\log\mu_{1-s}((\mathbb{I}-x)^{-1})ds+\int_K\log\mu_{s}(\mathbb{I}-x^*x) ds.\]
It follows from Lemma \ref{lemma 3.7}(3)-(5) and Lemma \ref{lemma3.6} that
\begin{align*}
\int_{K}\log\mu_{s}(A)ds&\geq
\int_{0}^{t}2\log\mu_{1-s}((\mathbb{I}-x)^{-1})ds+\int_K\log\mu_{s}(\mathbb{I}-x^*x) ds\\
&=\int_{0}^{t}2\log\frac{1}{\mu_{s}^\ell(\mathbb{I}-x)}ds+\int_K\log\mu_{s}(\mathbb{I}-x^*x) ds\\
&\geq\int_{0}^{t}2\log\frac{1}{1+\mu_{s}^\ell(x)}ds+\int_K\log(1-\mu_{1-s}^\ell(x)^2)ds\\
&=\int_{0}^{t}2\log\frac{1}{1+\mu_{s}(x)}ds+\int_K\log(1-\mu_{1-s}(x)^2)ds,
\end{align*}
because $\mu_{s}^\ell(x)=\mu_{s}(x)$ holds for almost every $t\in[0, 1].$
\end{proof}

\begin{corollary}\label{corollary 3.10}
Let $x\in \mathcal{M}$ with $\|x\|<1$. Then
\[\int_0^t\log\mu_{1-s}((\mathbb{I}-x^*)^{-1}(\mathbb{I}-x^*x)(\mathbb{I}-x)^{-1})ds\geq
\int_0^t\log\frac{1-\mu_s(x)}{1+\mu_s(x)}ds, ~t>0.\]
In particular,
\begin{align*}
\frac{\Delta_{\tau}(\mathbb{I}-x^*x)}{\Delta_{\tau}(\mathbb{I}-x)^2}
\geq\exp\int_0^1\log\frac{1-\mu_s(x)}{1+\mu_s(x)}ds.
\end{align*}
\end{corollary}
\begin{proof}
Replacing $K$ by $[1-t, 1]$, in Theorem \ref{theorem 3.10} we have
\begin{align*}
\int_{0}^{t}\log\mu_{1-s}(A)ds&=\int_{1-t}^{1}\log\mu_{s}(A)ds\\
&\geq
\int_{0}^{t}2\log\frac{1}{1+\mu_{s}(x)}ds+\int_{1-t}^1\log(1-\mu_{1-s}(x)^2)ds\\
&=\int_{0}^{t}2\log\frac{1}{1+\mu_{s}(x)}ds+\int_0^t\log(1-\mu_{s}(x)^2) ds\\
&=\int_{0}^{t}\log\frac{1-\mu_{s}(x)}{1+\mu_{s}(x)}ds.
\end{align*}
Therefore, letting $t\rightarrow1$ yields
\begin{align*}
\int_{0}^{1}\log\mu_{s}(A)ds&=\int_{0}^{1}\log\mu_{1-s}(A)ds
\geq\int_{0}^{1} \log\frac{1+\mu_{s}(x)}{1-\mu_{s}(x)}ds.
\end{align*}
This completes the proof.
\end{proof}

\begin{theorem}\label{theorem 3.11}
Let $0\leq x_i\in \mathcal{M}$ with $\|x_i\|<1, i=1, 2,\cdots, n$.
Then for any
unitary operator $u\in\mathcal{M}$ and positive scalars
$\omega_i, i=1, 2,\cdots, n,\sum^n_i\omega_i=1$, we have
\begin{align*}
\prod_{i=1}^n[\exp\int_0^1\log\frac{1-\mu_t(x_i)}{1+\mu_t(x_i)}dt]^{\omega_i}
\leq\frac{\Delta_\tau(\mathbb{I}-W^2)}{\Delta_\tau(\mathbb{I}-uW)^2}
\leq\prod_{i=1}^n[\exp\int_0^1\log\frac{1+\mu_t(x_i)}{1-\mu_t(x_i)}dt]^{\omega_i},
\end{align*}
where $W=\sum_{i=1}^n\omega_i x_i.$
\end{theorem}
\begin{proof}
An easy calculation shows that $1-W^2$ and $1-uW$ are invertible and $W\geq0$ with $\|W\|<1$.
Theorem \ref{theorem 3.9} and Corollary \ref{corollary 3.10} tell us that
\begin{equation}\label{inequ. theorem 3.11-1}
\exp\int_0^1\log\frac{1-\mu_t(x)}{1+\mu_t(x)}dt\leq
\frac{\Delta(\mathbb{I}-x^*x)}{\Delta(\mathbb{I}-x)^2}
\leq\exp\int_0^1\log\frac{1+\mu_t(x)}{1-\mu_t(x)}dt.
\end{equation}
Note that \cite[Theorem 4.4]{FK1986} tells us that
\[\int_0^t\mu_s(W)ds\leq \int_0^t\sum_{i=1}^n\omega_i\mu_s(x_i)ds.\]
The rest of the proof run as \cite[Theorem 5]{LZ2017}. For the convenience of the reader, we add a proof.
 Indeed, the convexity and the monotonicity of the function $f(t)=\log\frac{1+t}{1-t}, 0\leq t<1$ mean that
\[\int_0^tf(\mu_s(W))ds\leq \int_0^tf(\sum_{i=1}^n\omega_i\mu_s(x_i))ds.\]
On the other hand, by Lewent's inequality( \cite{LZ2017, L2013}), we obtain
\[\frac{1+\sum_{i=1}^n\omega_i\mu_s(x_i)}{1-\sum_{i=1}^n\omega_i\mu_s(x_i)}
\leq\prod_{i=1}^n(\frac{1+\mu_s(x_i)}{1-\mu_s(x_i)})^{\omega_i}.\]
Thus
\[\int_0^tf(\mu_s(W))ds\leq\int_0^t\log\prod_{i=1}^n(\frac{1+\mu_s(x_i)}{1-\mu_s(x_i)})^{\omega_i}ds
=\sum_{i=1}^n\omega_i\int_0^t\log(\frac{1+\mu_s(x_i)}{1-\mu_s(x_i)})ds.\]
It follows that
\begin{equation}\label{inequ. theorem 3.11-2}\exp\{\int_0^t\log(\frac{1+\mu_s(W)}{1-\mu_s(W)})ds\}\leq
\prod_{i=1}^n[\exp\int_0^1\log\frac{1+\mu_t(x_i)}{1-\mu_t(x_i)}dt]^{\omega_i}.
\end{equation}
Moreover, the inequalities in (\ref{inequ. theorem 3.11-2}) reverse by taking reciprocals, which implies
\begin{equation}\label{inequ. theorem 3.11-3}\exp\{\int_0^t\log(\frac{1-\mu_s(W)}{1+\mu_s(W)})ds\}\geq
\prod_{i=1}^n[\exp\int_0^1\log\frac{1-\mu_t(x_i)}{1+\mu_t(x_i)}dt]^{\omega_i}.
\end{equation}
Combining (\ref{inequ. theorem 3.11-1}) with (\ref{inequ. theorem 3.11-2}) and
(\ref{inequ. theorem 3.11-3}) yields
\begin{align*}
\prod_{i=1}^n[\exp\int_0^1\log\frac{1-\mu_t(x_i)}{1+\mu_t(x_i)}dt]^{\omega_i}
\leq\frac{\Delta_\tau(\mathbb{I}-W^2)}{\Delta_\tau(\mathbb{I}-uW)^2}
\leq\prod_{i=1}^n[\exp\int_0^1\log\frac{1+\mu_t(x_i)}{1-\mu_t(x_i)}dt]^{\omega_i}.
\end{align*}
\end{proof}

\section{Cayley transform with logarithmic submajorisation}
In this section, we will consider some logarithmic submajorisation
inequalities related to Cayley transform.
We will extend some results of Yang-\,-Zhang\cite{YZ2020}
to the case of finite von Neumann algebra.

Let $x\in\mathcal{M}$. If $x+i\mathbb{I}$ is invertible, we call
$\mathcal{C}(x)=(x-i\mathbb{I})(x+i\mathbb{I})^{-1}$ the Cayley transform of $x$.

\begin{theorem}\label{theorem 5.3}
Let $x, y\in\mathcal{M}$ with $\|x\|<1, \|y\|<1$ and let $\mathcal{C}(x)$ and
$\mathcal{C}(y)$ be the Cayley transforms of $x$ and $y$, respectively.
If $K$ is a Borel subset of $[0, 1]$ with $m(K)=t$ $(m(K)$ denotes the
Lebesgue measure of $K)$, then
\begin{align*}
&\int_K\log(1-\mu_{1-s}(x))ds-\int_0^t\log(1+\mu_{s}(x))ds\\
&\leq
\int_K\log\mu_s(\mathcal{C}(x))ds\\
&\leq\int_K\log(1+\mu_s(x))-\int_0^t\log(1-\mu_s(x))ds
\end{align*}
and
\begin{align*}
&\int_E\log\mu_s(\mathcal{C}(x)-\mathcal{C}(y))ds\\
\leq&
\int_E\log2\mu_s(x-y)ds-\int_0^t\log[(1-\mu_{s}(x))(1-\mu_{s}(y))]ds.
\end{align*}
\end{theorem}
\begin{proof}
Let us first compute the upper bounds. Remark \ref{remark 4.2} shows that
\begin{align*}
\int_K\log\mu_s(\mathcal{C}(x))ds&=\int_K\log\mu_s((x-i\mathbb{I})(x+i\mathbb{I})^{-1})ds\\
&\leq\int_K\log\mu_s(x-i\mathbb{I})ds
+\int_0^t\log\mu_s((x+i\mathbb{I})^{-1})ds.
\end{align*}
Together with Lemma \ref{lemma 3.7} this gives
\begin{align*}
\int_0^t\log\mu_s((x+i\mathbb{I})^{-1})ds&\leq \int_0^t\log[\mu_{1-s}^\ell(x+i\mathbb{I})]^{-1}ds\\
&\leq\int_0^t\log(1-\mu_{s}(x))^{-1}ds\\
&=-\int_0^t\log(1-\mu_{s}(x))ds.
\end{align*}
Thus
\begin{align*}
\int_K\log\mu_s(\mathcal{C}(x))ds
\leq\int_K\log(1+\mu_s(x))-\int_0^t\log(1-\mu_s(x))ds
\end{align*}

The lower bound follows easily by using Remark \ref{remark 4.2}.
Indeed,
from Remark \ref{remark 4.2} we obtain
\begin{align*}
\int_E\log\mu_s(\mathcal{C}(x))ds&\geq
\int_0^t\log\mu_{1-s}(x-i\mathbb{I})ds+\int_E\log\mu_{s}((x+i\mathbb{I})^{-1})ds\\
&\geq\int_0^t\log(1-\mu_{s}^\ell(x))ds-\int_E\log\mu_{1-s}^\ell(x+i\mathbb{I})ds\\
&\geq\int_0^t\log(1-\mu_{s}^\ell(x))ds-\int_E\log(1+\mu_{s}^\ell(x))ds\\
&=\int_0^t\log(1-\mu_{s}(x))ds-\int_E\log(1+\mu_{s}(x))ds.
\end{align*}

For the second part, an easy calculation shows that
$\mathcal{C}(x)=1-2i(x+i\mathbb{I})^{-1}$ and
\[
\mathcal{C}(x)-\mathcal{C}(y)=2i(y+i\mathbb{I})^{-1}(x-y)(x+i\mathbb{I})^{-1}.
\]
Hence, Remark \ref{remark 4.2} implies that
\begin{align*}
\int_E\log\mu_s(\mathcal{C}(x)-\mathcal{C}(y))ds=&
\int_E\log2\mu_s((y+i\mathbb{I})^{-1}(x-y)(x+i\mathbb{I})^{-1})ds\\
\leq& \int_0^t\log\mu_s((y+i\mathbb{I})^{-1})ds+\int_E\log2\mu_s(x-y)ds\\
&+\int_0^t\log\mu_s((x+i\mathbb{I})^{-1})ds\\
\leq& \int_0^t\log[\mu_{1-s}^\ell(y+i\mathbb{I})]^{-1} ds+\int_E\log2\mu_s(x-y)ds\\
&+\int_0^t\log[\mu_{1-s}^\ell(x+i\mathbb{I})]^{-1}ds\\
\leq& \int_0^t\log[1-\mu_{s}(y)]^{-1} ds+\int_E\log2\mu_s(x-y)ds\\
&+\int_0^t\log[1-\mu_{s}(x)]^{-1}ds\\
=&\int_E\log2\mu_s(x-y)ds\\
&-\int_0^t\log[(1-\mu_{s}(x))(1-\mu_{s}(y))]ds.
\end{align*}
\end{proof}

If we replace $K$ by $[0, 1]$, in Theorem \ref{theorem 5.3} we have the following corollary.
\begin{corollary}
Let $x, y\in\mathcal{M}$ with $\|x\|<1, \|y\|<1$ and let $\mathcal{C}(x)$ and
$\mathcal{C}(y)$ be the Cayley transforms of $x$ and $y$, respectively.
Then
\begin{align*}
 \int_0^1\log\frac{1-\mu_{1-s}(x)}{1+\mu_{s}(x)}ds
 \leq
\int_0^1\log\mu_s(\mathcal{C}(x))ds \leq\int_0^1\log\frac{1+\mu_s(x)}{1-\mu_s(x)}ds
\end{align*}
and
\begin{align*}
\int_0^1\log\mu_s(\mathcal{C}(x)-\mathcal{C}(y))ds\leq
\int_0^1\log\frac{2\mu_s(x-y)}{(1-\mu_{s}(x))(1-\mu_{s}(y))}ds.
\end{align*}
\end{corollary}


\section*{ Acknowledgments}

The first author wishes to express his thanks to Dr.Minghua Lin for suggesting the problem
and for many stimulating conversations.

\section*{ Funding}
This research was partially supported by the National Natural Science Foundation of
China No. 11761067 and National Natural Science Foundation of China No. 11801486.

\section*{Availability of data and materials}

Not applicable.

\section*{ Competing interests}
The author declares that there is no conflict of interests regarding the publication of this paper.

\section*{ Authors's contributions }
Both authors contributed equally and significantly in writing this article.
Both authors read and approved the final manuscript.


\bibliographystyle{amsplain}

\end{document}